\documentclass[a4paper,12pt]{amsart}

\usepackage{amsmath,amsthm,amsfonts,amssymb,stmaryrd,mathrsfs,enumitem}
\usepackage{latexsym}
\usepackage{fullpage}
\usepackage{a4,color,palatino,fancyhdr}
\usepackage{graphicx}%poner [dvips] para grficos
\usepackage{float}  %Este packete permite controlar la flotacion de figuras al poner [H] o [h] como opcion de \begin{figure}
\usepackage[small, bf, margin=90pt, tableposition=bottom]{caption}
\RequirePackage{amssymb}
\RequirePackage[T1]{fontenc}
\usepackage{amscd}
\usepackage{xypic}
\input{xy}
\xyoption{all}
\xyoption{poly}
\usepackage[all]{xy}
\usepackage{tikz}
\newtheorem{theorem}{Theorem}[section]

\newtheorem{lemma}[theorem]{Lemma}
\newtheorem{corollary}[theorem]{Corollary}
\newtheorem{remark}{Remark}

\newtheorem{con}[theorem]{Conjecture}

\numberwithin{equation}{section}

\title[A proof of Kosniowski conjecture]{A proof of Kosniowski conjecture}
\author{Zhi L\"{u}}

\address{School of Mathematical Sciences, Fudan University, Shanghai 200433, P. R. China}
\email{zlu@fudan.edu.cn}

\thanks{The author is partially supported by the NSFC grant No. 11971112.}

\begin{document}

%%% ----------------------------------------------------------------------

\begin{abstract}
Let  $M$ be a unitary $S^1$-manifold with only isolated fixed points such that $M$ is not a boundary.
We show that $4\chi(M)>\dim M$, where $\chi(M)$ is the Euler characteristic of $M$. This gives an affirmative answer  of  Kosniowski conjecture.

\end{abstract}
%%% ----------------------------------------------------------------------
\maketitle

\section{Introduction}
A unitary manifold is a smooth closed manifold with a stablely complex structure on its tangent bundle.
With respect to smooth $S^1$-actions on unitary manifolds, there is a standing more than 40-year-old conjecture, posed by Kosniowski in his 1979's paper~\cite[Conjecture A]{Ko2}. Specifically speaking,  suppose that $M$ is a unitary $S^1$-manifold with isolated fixed points such that $M$ is not a boundary.
  Kosniowski conjectured that  the number of fixed points is greater than $f(\dim M)$ where $f$ is some (linear) function. He further noted that  the most likely function is $f(x)={x\over 4}$.
    We know from \cite[Chapter 1, (1.5)]{AP} that
    $\chi(M)=\chi(M^{S^1})$ where $\chi(\cdot)$ denotes  the Euler characteristic and $M^{S^1}$ is the fixed point set. So $\chi(M)$ is equal to the number of  fixed points. With this result together,  Kosniowski conjecture is stated  as follows.
 \begin{con}[Kosniowski]\label{conj}
 Suppose that $M$ is a nonbounding unitary $S^1$-manifold with isolated fixed points. Then
 $$4\chi(M)>\dim M.$$
 \end{con}

\begin{remark}
  It is well-known that if $M$ is a unitary $S^1$-manifold with isolated fixed points, then $\dim M$ must be even. Write $\dim M=2n$, then Conjecture~\ref{conj} becomes $2\chi(M)>n$.
\end{remark}
 Unitary manifolds form a class of nice behaved geometric objects, including  complex manifolds, almost complex manifolds and symplectic manifolds etc. Thus, Kosniowski conjecture will probably play an important role on the various areas, such as topology,  complex geometry,   symplectic geometry, transformation groups and so on.

\vskip .2cm
When $\chi(M)=2$,  in his thesis Kosniowski first proved the conjecture and determined that the possible value of $n$ is $1$ or $3$ (also see~\cite{Ko1, Ko2}).  Under the condition that the $S^1$-action on $M^{2n}$ can be extended to the $T^n$-action on $M^{2n}$, L\"u--Tan in \cite{LT} proved the conjecture. Ma--Wen in \cite{MW} also discussed the case in which the $S^1$-action on $M^{2n}$ can be extended to the $T^{n-1}$-action on $M^{2n}$.  When $\chi(M)=3$ and $M$ is assumed to be almost complex, Jang showed in \cite{J3} that only $n=2$ happens, so the conjecture holds in this case.
In the setting of compact symplectic manifolds with symplectic circle actions fixing isolated points, Pelayo--Tolman showed in \cite{PT} that if the weights satisfy some subtle condition, then the action has at least $n+1$ isolated fixed points. In the setting of almost complex manifolds with circle actions,  Li--Liu showed in \cite{LL} that if  an almost complex manifold $M$ has some nonzero Chern number, then  any $S^1$-action on $M$ must have at least $n+1$ fixed points. All these provide the evidence to support the conjecture.

\vskip .2cm
The main tool  used here is the localization formula of rigid $T_{x,y}$-genus of $M$ in terms of the fixed point data, essentially due to Krichever~\cite{K}, which is a generalization of Atiyah--Hirzebruch formula for the $\chi_y$-genus of a complex $S^1$-manifold (\cite{AH}).
We will see that this formula plays an essential role on our work.
  We first read out a system of equations from the localization formula of the $T_{x,y}$-genus of $M$. Then using Newton's formula, we further  change this system into a new system such that its coefficient matrix is a Vandemonde matrix, from which we induce some cyclic equalities  if $\chi(M)< [{n\over 2}]+1$, giving a contradiction. This establishes  $\chi(M)\geq [{n\over 2}]+1>{n\over 2}$, thus obtaining the affirmative answer  of  Kosniowski conjecture.  The result is stated as follows.

 \begin{theorem}\label{main}
 Suppose that $M$ is a nonbounding unitary $S^1$-manifold with isolated fixed points. Then $4\chi(M)>\dim M$.
 \end{theorem}

Making  use of the localization formula of rigid $L$-genus, our approach is automatically   applied to the case of oriented $S^1$-manifolds, so we also have that

\begin{theorem}\label{main1}
 Suppose that $M$ is a nonbounding oriented closed smooth $S^1$-manifold with isolated fixed points. Then $2\chi(M)>\dim M$.
 \end{theorem}

 This note is organized as follows. In Section~\ref{material}  we review the relative work with respect to genera and rigidity, especially for $T_{x,y}$-genus and its rigidity, and state a useful lemma.
 In Section~\ref{proof} we give the proof of Theorem~\ref{main}.
 Finally we carry out a simple investigation on the relation between the Euler characteristic and the top Chern number of a unitary $S^1$-manifold   in Section~\ref{ob}.

\section{The  $T_{x,y}$-genus of unitary $S^1$-manifolds} \label{material}

First let us review the relative work with respect to genera and rigidity.

\subsection{Genera and rigidity}
In his seminar book~\cite{H}, Hirzebruch
 introduced various multiplicative genera. A genus which is an invariant of unitary bordism classes can be defined  by a homomorphism $\phi:  \Omega_*^U\longrightarrow R$ where $\Omega_*^U$ is the complex bordism ring
 and  $R$ is a commutative ring with unit. If we assume that $R$ has no additive torsion, then $\phi$ is completely determined by
 $\phi\otimes\mathbb{Q}: \Omega_*^U\otimes\mathbb{Q}\longrightarrow R\otimes\mathbb{Q}$. Using universal $R$-valued characteristic classes of special type,  Hirzebruch gave the explicit description of $\phi\otimes\mathbb{Q}$.  For any multiplicative genus $\phi$, there always exists a multiplicative sequence
 $\{K_i(c_1, ..., c_i)\}$ of
polynomials given by a characteristic series such that $\phi(M)=K_n(c_1, ..., c_n)[M]$ where
$c_i$ is the $i$-th Chern class of a unitary manifold $M$ of $\dim_{\mathbb{C}}M=n$ and $[M]$ is the fundamental homology class of $M$.
%\begin{remark} \end{remark}
 \vskip .2cm
 It is well-known  that a multiplicative genus $\phi: \Omega_*^U\longrightarrow R$
 can induce an
equivariant genus $\phi^G: \Omega_*^{U, G}\longrightarrow K(BG)\otimes R$,
 where $\Omega_*^{U, G}$ is  the ring of complex bordisms of unitary manifolds with actions of a connected compact
Lie group $G$.
A multiplicative genus $\phi$ is called {\em rigid} if $\phi^G(M)=\phi(M)$
for any connected compact group $G$, i.e., $\phi(M)\in R$.
  The rigidity for $S^1$ implies the rigidity for  $G$ (also see \cite{AH, BP, K}),
so the former one is more essential. When $G=S^1$, the universal classifying space $BS^1$ is $\mathbb{C}P^\infty$, so the ring $K(BS^1)\otimes R$
is isomorphic to the ring of formal power series $R[[u]]$, where  $R$ is often chosen as $\mathbb{Q}$.
\vskip .2cm
In \cite{AH},  Atiyah--Hirzebruch  proved
that $L$-genus (i.e., signature)  is rigid for oriented $S^1$-manifolds and $\chi_y$-genus is rigid
for  complex $S^1$-manifolds, and gave the corresponding localization formulae. In \cite{K},  Krichever
introduced the equivariant $T_{x,y}$-genus. He proved the rigidity of  $T_{x,y}$-genus of an almost $S^1$-manifold $M$ and established the localization formula of  $T_{x,y}$-genus of $M$.
As noted in \cite[Page 382, Remark]{BP}, Krichever carried out his work by implicitly assuming manifolds to be almost complex. However, actually his proof automatically extends to the unitary case. For more details, see~\cite{BP,  BPR, K}. In addition, Musin also showed in \cite{M} that if a genus is rigid, then it is $T_{x,y}$-genus. In the view points of complex analysis and number-theoretic abstraction, $T_{x,y}$-genus
was further discussed in \cite{LM}.

\subsection{Localization formula of $T_{x,y}$-genus} \label{local}

Throughout the following, assume that $M$ is a unitary manifold with an action of $S^1$ preserving the unitary structure and fixing only isolated points $p_1, ..., p_m$. Choose a fixed point $p_i$, the tangent space $\mathcal{T}_{p_i}M$ at $p_i$ is a complex $S^1$-representation, so this forces $\dim M=\dim \mathcal{T}_{p_i}M$ to be even.
 Let $\dim M=2n$. Then $\mathcal{T}_{p_i}M$ can be decomposed into the sum of $n$ irreducible $S^1$-representations
$$\mathcal{T}_{p_i}M=\bigoplus_{j=1}^n T_{p_i, j}$$
where  the action of $S^1$ on each $T_{p_i,j}$ is given  by
$(g, z)\longmapsto g^{w_{i,j}}z$
for  $z\in T_{p_i,j}$, $g\in S^1$ and some nonzero integer $w_{i,j}$. The collection $\{w_{i,1}, ..., w_{i, n}\}$, denoted by ${\bf w}_i$,  gives the weight set of the complex $S^1$-representation $\mathcal{T}_{p_i}M$.
The orientation of $M$ and the orientation of the unitary structure on the tangent bundle $\mathcal{T}M$ naturally induce two orientations on $\mathcal{T}_{p_i}M$. Thus, the equivariant Euler class of the normal bundle $\nu_{p_i}$ to $p_i$ in $M$ is
$$\chi^{S^1}(\nu_{p_i})=\varepsilon_i\chi^{S^1}(\mathcal{T}_{p_i}M)$$
where $\varepsilon_i=+1$ if two orientations agree, and $\varepsilon_i=-1$ otherwise.

\vskip .2cm
Let $T_{x,y}(M)$ be the $T_{x,y}$-genus of $M$. Its $S^1$-equivariant version is denoted by $T^{S^1}_{x,y}(M)$,  corresponding to the characteristic series
$$H_{x,y}(u)={{u(xe^{u(x+y)}+y)}\over {e^{u(x+y)}-1}}.$$
As mentioned as before, the $T_{x,y}$-genus of $M$ is rigid, so $T_{x,y}(M)=T^{S^1}_{x,y}(M)$.
 Then the localization formula of $T^{S^1}_{x,y}(M)$ is
 $$T^{S^1}_{x,y}(M)=\sum_{i=1}^m \varepsilon_i\prod_{j=1}^n{{H_{x,y}(w_{i,j}u)}\over {w_{i,j}u}}$$
Write $q^{-1}=e^{(x+y)u}$. Then $T^{S^1}_{x,y}(M)$ becomes
\begin{equation}\label{Laurant}
T^{S^1}_{x,y}(M)=\sum_{i=1}^m\varepsilon_i\prod_{j=1}^n\frac{x+yq^{w_{i,j}}}{1-q^{w_{i,j}}}\in
\mathbb{Z}[x, y][[q]]
\end{equation}
due to  Krichever (also see \cite[Theorem 9.4.8]{BP}). Let ${\bf w}_i^+$ (resp. ${\bf w}_i^-$) denote the collection of all positive weights (resp. all negative weights) in ${\bf w}_i$ at the fixed point $p_i$.
 By $|{\bf w}_i^\pm|$ we mean the number of weights in ${\bf w}_i^\pm$, so $|{\bf w}_i^+|+|{\bf w}_i^-|=n$.  Furthermore, as $q\longrightarrow 0$,  we have that
\begin{theorem}[Generalized Atiyah--Hirzebruch formula~\cite{K}]\label{K-formula}
$$T_{x,y}(M)=T^{S^1}_{x,y}(M)=\sum_{i=1}^m\varepsilon_i\prod_{j=1}^n\frac{x+yq^{w_{i,j}}}{1-q^{w_{i,j}}}
=\sum_{i=1}^m\varepsilon_ix^{|{\bf w}_i^+|}(-y)^{|{\bf w}_i^-|}\in \mathbb{Z}[x, y].$$
\end{theorem}
\begin{remark}
As $q\longrightarrow\infty$, there is also another expression of  $T_{x,y}(M)$ $$T_{x,y}(M)=\sum_{i=1}^m\varepsilon_ix^{|{\bf w}_i^-|}(-y)^{|{\bf w}_i^+|}.$$
\end{remark}
Clearly $T_{x,y}(M)=T^{S^1}_{x,y}(M)$  does not depend on $q$ and (\ref{Laurant}) is actually a Laurant polynomial in $q$.
\vskip .2cm

Combining with the same terms in $\sum\limits_{i=1}^m\varepsilon_ix^{|{\bf w}_i^+|}(-y)^{|{\bf w}_i^-|}$, we may write
$$ \sum_{i=1}^m\varepsilon_ix^{|{\bf w}_i^+|}(-y)^{|{\bf w}_i^-|}=\sum_{l=0}^n \chi^l(M) x^{n-l}y^{l}$$
where each $\chi^l(M)$ is an integer.
Compare with two sides of the following equation
$$\sum_{i=1}^m\varepsilon_i\prod_{j=1}^n\frac{x+q^{w_{i,j}}y}{1-q^{w_{i,j}}}=\sum_{l=0}^n \chi^l(M) x^{n-l}y^{l},$$
an easy observation gives the following formula
\begin{equation}\label{chi-e}
\chi^l(M)=\sum_{i=1}^m\varepsilon_i\frac{\sigma_l(q^{w_{i,1}},..., q^{w_{i,n}})}{\prod\limits_{j=1}^n(1-q^{w_{i,j}})}
\end{equation}
for $0\leq l\leq n$, where $\sigma_l$ is the $l$-th elementary symmetric function. As far as the author knows, the formula (\ref{chi-e}) without signs $\varepsilon_i$ first appeared in \cite[(2.3) of Theorem, page 45]{Ko0} and was used very well in \cite[page 172]{Ko1}.

\begin{remark}
When $x=1$, it is well-known that $T_{1,y}(M)$ is exactly the Atiyah--Hirzebruch $\chi_y$-genus
$$T_{1,y}(M)=\chi_y(M)=\sum_{l=0}^n \chi^l(M) y^{l}.$$
Moreover,
\begin{enumerate}
\item The signature of $M$ is
$
\mathrm{Sign}(M)=T_{1,1}(M)=\sum\limits_{l=0}^n \chi^l(M);
$
\item The top Chern number of $M$ is
$
c_n[M]=T_{1,-1}(M)=\sum\limits_{l=0}^n(-1)^l\chi^l(M);
$
\item
The Todd-genus of $M$ is
$\text{Todd}(M)=T_{1,0}(M)=\chi^0(M).$
\end{enumerate}
\end{remark}

\vskip .2cm
The following result  can be found in \cite{HBJ} for a complex manifold and in \cite{L} for an almost complex manifold. Actually it can be extended to the unitary situation automatically. Here a simple proof is included for a local completeness.

\begin{lemma} \label{sy-equ}
For $0\leq l\leq n$,
$$\chi^l(M)=(-1)^n\chi^{n-l}(M).$$
\end{lemma}
\begin{proof}
Using the rigidity of $T_{x,y}$-genus and replacing $q$ by ${1\over q}$, the formula (\ref{chi-e}) becomes
\begin{align*}
\chi^l(M)=& \sum_{i=1}^m\varepsilon_i\frac{\sigma_l(q^{w_{i,1}},..., q^{w_{i,n}})}{\prod\limits_{j=1}^n(1-q^{w_{i,j}})}
= \sum_{i=1}^m\varepsilon_i\frac{\sigma_l(q^{-w_{i,1}},..., q^{-w_{i,n}})}{\prod\limits_{j=1}^n(1-q^{-w_{i,j}})}\\
=& \sum_{i=1}^m\varepsilon_i\frac{\sigma_l(q^{-w_{i,1}},..., q^{-w_{i,n}})}
{(-1)^n \big(\prod\limits_{j=1}^nq^{-w_{i,j}}\big)\prod\limits_{j=1}^n(1-q^{w_{i,j}})}\\
=& (-1)^n\sum_{i=1}^m\varepsilon_i\frac{{{\sigma_l(q^{-w_{i,1}},..., q^{-w_{i,n}})}\over{\big(\prod\limits\limits_{j=1}^nq^{-w_{i,j}}\big)}}}
{\prod\limits_{j=1}^n(1-q^{w_{i,j}})}=(-1)^n\sum_{i=1}^m\varepsilon_i\frac{\sigma_{n-l}(q^{w_{i,1}},..., q^{w_{i,n}})}{\prod\limits_{j=1}^n(1-q^{w_{i,j}})}\\
=& (-1)^n\chi^{n-l}(M)
\end{align*}
as desired.
\end{proof}

\begin{corollary}
If $n$ is odd, then $\mathrm{Sign}(M)=0$.
\end{corollary}

\section{Proof of Theorem~\ref{main}}\label{proof}
Let $M^{2n}$ be a unitary manifold with an action of $S^1$ preserving the unitary structure and fixing only isolated points $p_1, ..., p_m$. Suppose further that $M^{2n}$ is not a boundary.
Following the notations of Section~\ref{material}, let $f_i(q)={1\over {\prod\limits_{j=1}^n(1-q^{w_{i,j}})}}$ for $1\leq i\leq m$.
 Then  for $0\leq l\leq n$, we rewrite the formula (\ref{chi-e}) into the following
 \begin{align}\label{chi-ec}
\chi^l(M)=\sum_{i=1}^m\varepsilon_i\sigma_l(q^{w_{i,1}},..., q^{w_{i,n}})f_i(q).
\end{align}
We know by Lemma~\ref{sy-equ} that $\chi^l(M)=(-1)^n\chi^{n-l}(M)$. This implies that essentially there are  $[{n\over 2}]+1$ different formulae of the form (\ref{chi-ec}), where $$[{n\over 2}]+1=\begin{cases}
 k+1 & \text{if $n=2k$}\\
 k+1 & \text{if $n=2k+1$}. \end{cases}$$
 Thus it suffices to consider $k+1$ formulae of $\chi^l(M), l=0,1, ..., k$.
  This gives the following system of $k+1$  equations:
 \begin{equation}\label{e-system1}
 \begin{cases}
 \sum\limits_{i=1}^m \varepsilon_i\sigma_0(q^{w_{i,1}}, ..., q^{w_{i,n}})f_i(q) =&\chi^0(M) \\
 \sum\limits_{i=1}^m \varepsilon_i\sigma_1(q^{w_{i,1}}, ..., q^{w_{i,n}})f_i(q)= &\chi^1(M)\\
 \ \ \ \ \ \ \ \ \ \ \ \ \vdots & \ \ \ \ \vdots\\
 \sum\limits_{i=1}^m\varepsilon_i\sigma_{k}(q^{w_{i,1}}, ..., q^{w_{i,n}})f_i(q) =& \chi^k(M).
 \end{cases}
 \end{equation}
 Note that $\sigma_0(q^{w_{i,1}}, ..., q^{w_{i,n}})=1$ for all $i$.
  By the Newton's formula (see~\cite[Problem 16-A]{MS})
 $$s_l-\sigma_1s_{l-1}+\cdots+(-1)^{l-1} \sigma_{l-1}s_1 +(-1)^ll\sigma_l=0,$$
we can induce
 $$l\sigma_l=(-1)^l\sigma_1^l+F_l$$
 for  $l\geq 2$, where $F_l$ is a symmetric polynomial.
 For example, when $l=2$, we have $2\sigma_2=\sigma_1^l-s_2$ so $F_2=-s_2$. When $l=3$, we have $$3\sigma_3=(-1)^3\sigma_1^3+3\sigma_1\sigma_2+s_3$$ so $F_3=3\sigma_1\sigma_2+s_3$. Then   for $l\geq 2$,
 \begin{align*}
 l\chi^l(M)=&\sum\limits_{i=1}^m \varepsilon_il\sigma_{l}(q^{w_{i,1}}, ..., q^{w_{i,n}})f_i(q)\\
 =&  \sum\limits_{i=1}^m \varepsilon_i(-1)^{l}\sigma_1^l(q^{w_{i,1}}, ..., q^{w_{i,n}})f_i(q)+\sum\limits_{i=1}^m \varepsilon_iF_l(q)f_i(q)\\
 \end{align*}
 where $F_l(q)$ is a certain function in $q$.
  Set $\mathcal{F}_l(q)=\sum\limits_{i=1}^m \varepsilon_iF_l(q)f_i(q)$.
 Then we can further change the  system (\ref{e-system1}) of equations into
  \begin{equation}\label{e-system2}
 \begin{cases}
 \sum\limits_{i=1}^m \varepsilon_if_i(q)&= \chi^0(M) \\
 \sum\limits_{i=1}^m \varepsilon_i\sigma_1(q^{w_{i,1}}, ..., q^{w_{i,n}})f_i(q)&=\chi^1(M)\\
 \sum\limits_{i=1}^m \varepsilon_i\sigma_1^2(q^{w_{i,1}}, ..., q^{w_{i,n}})f_i(q)& = 2\chi^2(M)-\mathcal{F}_2(q)\\
 \ \ \ \ \ \ \ \ \ \ \ \ \ \ \vdots & \  \vdots\\
 \sum\limits_{i=1}^m \varepsilon_i\sigma_1^k(q^{w_{i,1}}, ..., z^{w_{i,n}})f_i(q)&=(-1)^k\big(k\chi^k(M)-\mathcal{F}_k(q)\big).
 \end{cases}
 \end{equation}
% where $\mathcal{F}_0(q)=\mathcal{F}_1(q)=0$.
 Obviously, we can regard
$${\bf x}(q)=\big(\varepsilon_1f_1(q), \varepsilon_2f_2(q), ..., \varepsilon_mf_m(q)\big)^\top$$  as a  nonzero solution of the equation system (\ref{e-system2}) with the following coefficient matrix
 \begin{gather*}
 A(q)=
\begin{pmatrix}
   1 &\cdots & 1      \\
\sigma_1(q^{w_{1,1}}, ..., q^{w_{1,n}})    & \cdots & \sigma_1(q^{w_{m,1}}, ..., q^{w_{m,n}})        \\
\vdots&  &  \vdots   \\
\sigma_1^{k}(q^{w_{1,1}}, ..., q^{w_{1,n}})   &   \cdots     & \sigma_1^{k}(q^{w_{m,1}}, ..., q^{w_{m,n}})  \\
\end{pmatrix}.
\end{gather*}
For a convenience, we simply write
$$\sigma_{1, i}=\sigma_1(q^{w_{i,1}}, ..., q^{w_{i,n}}) \ \text{for}\  i=1, ..., m$$
and
\begin{align*}
{\bf b}(q)=& (b_0(q), b_1(q), b_2(q), ..., b_k(q))^\top\\
= &\Big(\chi^0(M), \chi^1(M),  2\chi^2(M)-\mathcal{F}_2(q), ..., (-1)^k\big(k\chi^k(M)-\mathcal{F}_k(q)\big)\Big)^\top.
\end{align*}
Then the system (\ref{e-system2}) can be  written as the following simple form
\begin{equation}\label{e-system2-1}
A(q){\bf x}(q)={\bf b}(q).
\end{equation}

Next we are going to prove the following key lemma.

\begin{lemma}\label{key}
It is impossible that $k+1>m$.
\end{lemma}
\begin{proof}
Suppose that $k+1>m$. Let us look at the system (\ref{e-system2-1}). Since  $M^{2n}$ has been assumed to be not a boundary, without a loss of generality, we may assume that all %$\sigma_1(q^{w_{i,1}}, ..., q^{w_{i,n}}), 
$\sigma_{1, i}, i=1, ..., m$ are mutually  distinct. In fact, if there are some two %$\sigma_1(q^{w_{i_1,1}}, ..., q^{w_{i_1,n}})$ and  $\sigma_1(q^{w_{i_2,1}}, ..., q^{w_{i_2,n}})$ 
$\sigma_{1, i_1}$ and $\sigma_{1, i_2}$ that are both equal, then
their weight sets are the same, too. When $\varepsilon_{i_1}\not=\varepsilon_{i_2}$, we can do a simple surgery to delete two fixed points $p_{i_1}$ and $p_{i_2}$. This will unchange the $S^1$-manifold $M$ up to unitary bordism.
When $\varepsilon_{i_1}=\varepsilon_{i_2}$,
 we can move those terms corresponding to the $i_2$-th  column  of $A(q)$  in the left side of the system (\ref{e-system2-1}) into its right side. Whichever of above cases happens, we can always modify the system (\ref{e-system2-1})  into a new system, satisfying that  all %$\sigma_1(q^{w_{i,1}}, ..., q^{w_{i,n}})$'s
 $\sigma_{1, i}$'s appearing in the new coefficient matrix  are distinct. Actually this will reduce the number of
  polynomials in ${\bf x}(q)$ to be smaller, but this smaller number can not become zero since $M^{2n}$ is nonbounding. So the above distinct assumption for all %$\sigma_1(q^{w_{i,1}}, ..., q^{w_{i,n}})$'s 
 $\sigma_{1, i}$'s  will not produce any essential influence on our further argument.

 \vskip .2cm
 Since $k+1>m$, we have $k+1\geq m+1$.
Without a loss of generality, we may assume that $k=m$, so the system (\ref{e-system2-1}) is formed by $m+1$ equations. Then the system (\ref{e-system2-1}) can produce  $m+1$ systems, each of which consists of $m$ equations, stated as follows:
\begin{equation}\label{e-system2-2}
V_l(q){\bf x}(q)={\bf b}_l(q)
\end{equation}
for $l=0, 1, ..., m$,  where $${\bf b}_l(q)=\big(b_0(q), b_1(q), ..., b_{l-1}(q), \widehat{b_{l}(q)}, b_{l+1}(q), ..., b_m(q)\big)^\top$$ means that the $(l+1)$-th element is deleted in ${\bf b}(q)$, and \begin{gather*}
 V_l(q)=
\begin{pmatrix}
   1 &\cdots & 1& \cdots & 1      \\
\sigma_{1,1}    & \cdots & \sigma_{1,i} &\cdots& \sigma_{1,m}        \\
\vdots& \cdots & &\cdots& \vdots   \\
\sigma^{l-1}_{1,1}    & \cdots & \sigma^{l-1}_{1,i} &\cdots& \sigma^{l-1}_{1,m}        \\
\sigma^{l+1}_{1,1}    & \cdots & \sigma^{l+1}_{1,i}& \cdots& \sigma^{l+1}_{1,m}        \\
\vdots& \cdots && \cdots& \vdots   \\
\sigma_{1,1}^{m}   &   \cdots &\sigma_{1,i}^{m}  &\cdots   & \sigma_{1,m}^{m}  \\
\end{pmatrix}
\end{gather*}
is  a  Vandermonde matrix obtained by deleting the $(l+1)$-th row of $A(q)$.
It is easy to see that ${\bf b}_0(q), {\bf b}_1(q), ..., {\bf b}_m(q)$ are linearly dependent.

\vskip .2cm
Now suppose that
\begin{equation}\label{line}
u_0(q){\bf b}_0(q)+u_1(q){\bf b}_1(q)+\cdots+u_m(q){\bf b}_m(q)={\bf 0}
\end{equation}
where all $u_l(q)$'s are  functions in $q$.
More explicitly,  (\ref{line}) can be written as
\begin{equation}\label{line-a}
\begin{cases}
u_0(q)b_1(q)+ \big(u_1(q)+\cdots +u_m(q)\big)b_0(q) &=0 \\
\big(u_0(q)+u_1(q)\big)b_2(q)+\big(u_2(q)+\cdots+u_m(q)\big)b_1(q) &=0\\
\ \ \ \ \vdots &\  \vdots\\
 \big(u_0(q)+\cdots +u_{m-2}(q)\big)b_{m-1}(q)+\big(u_{m-1}(q)+u_m(q)\big)b_{m-2}(q)&=0 \\
 \big(u_0(q)+\cdots+u_{m-2}(q) +u_{m-1}(q)\big)b_m(q)+u_m(q)b_{m-1}(q)&=0.
\end{cases}
\end{equation}
Our next task is to determine the  precise expressions of all functions $u_l(q)$'s. We proceed our argument as follows.

\vskip .2cm
Combining with (\ref{e-system2-2}) and (\ref{line}), we have that
\begin{equation}\label{line1}
\Big(\sum_{l=0}^mu_l(q)V_l(q)\Big){\bf x}(q)={\bf 0}
\end{equation}
whose coefficient matrix $\sum\limits_{l=0}^mu_l(q)V_l(q)$
is explicitly written as follows:
\begin{align*}\label{coeff}
%\sum\limits_{l=0}^mu_l(q)V_l(q)=
{\tiny \begin{pmatrix}
\sigma_{1,1}u_0(q)+\sum\limits_{j=1}^mu_j(q) & \sigma_{1,2}u_0(q)+\sum\limits_{j=1}^mu_j(q) & \cdots &\sigma_{1,m}u_0(q)+\sum\limits_{j=1}^mu_j(q)   \\
 \sigma^2_{1,1}\sum\limits_{j=0}^1u_j(q)+\sigma_{1,1}\sum\limits_{j=2}^mu_j(q) &   \sigma^2_{1,2}\sum\limits_{j=0}^1u_j(q)+\sigma_{1,2}\sum\limits_{j=2}^mu_j(q)& \cdots &  \sigma^2_{1,m}\sum\limits_{j=0}^1u_j(q)+\sigma_{1,m}\sum\limits_{j=2}^mu_j(q)     \\
 \vdots  &\vdots & &\vdots & \\
 \sigma^l_{1,1}\sum\limits_{j=0}^{l-1}u_j(q)+\sigma^{l-1}_{1,1}\sum\limits_{j=l}^mu_j(q) & \cdots &  & \sigma^l_{1,m}\sum\limits_{j=0}^{l-1}u_j(q)+\sigma^{l-1}_{1,m}\sum\limits_{j=l}^mu_j(q)   \\
 \vdots  &\vdots & \ddots&\vdots & \\
\sigma^m_{1,1}\sum\limits_{j=0}^{m-1}u_j(q)+\sigma^{m-1}_{1,1}u_m(q) & \cdots &  & \sigma^m_{1,m}\sum\limits_{j=0}^{m-1}u_j(q)+\sigma^{m-1}_{1,m}u_m(q)    \\
\end{pmatrix}.}
\end{align*}
In order to give the  precise expressions of all functions $u_l(q)$'s, we consider $m$ diagonal  elements of the matrix $\sum\limits_{l=0}^mu_l(q)V_l(q)$, which are expressed in $u_l(q)$.
We let these $m$ expressions in $u_l(q)$ become zero, so that  we obtain  the following system of $m$ equations
\begin{equation}\label{diag}
W(q){\bf u}(q)={\bf 0}
\end{equation}
where ${\bf u}(q)=\big(u_0(q), u_1(q), ..., u_m(q)\big)^\top$ and the coefficient matrix
\begin{align*}\label{coeff}
W(q)=
\begin{pmatrix}
   \sigma_{1,1} & 1 & 1 & \cdots & \cdots &1  & 1    \\
 \sigma_{1,2} &  \sigma_{1,2}& 1 &\cdots & \cdots & 1 & 1      \\
  \sigma_{1,3} &  \sigma_{1,3}& \sigma_{1,3} & 1 &\cdots & 1  & 1      \\
\vdots  &\vdots &\ddots &\ddots &\ddots &\vdots &\vdots \\
\sigma_{1,m-1} & \cdots &  \cdots& \sigma_{1,m-1}&\sigma_{1,m-1} & 1 & 1 \\
\sigma_{1,m} & \cdots &  \cdots &  \cdots  &  \sigma_{1,m} &\sigma_{1,m} & 1  \\
\end{pmatrix}.
\end{align*}
By a direct calculation, we obtain that
\begin{equation}\label{result1}
\begin{cases}
u_1(q) & ={{\sigma_{1,1}-\sigma_{1,2}}\over {\sigma_{1,2}-1}}u_0(q)\\
u_1(q)+u_2(q) &= {{\sigma_{1,1}-\sigma_{1,3}}\over {\sigma_{1,3}-1}}u_0(q)\\
\ \ \ \ \vdots &\  \vdots\\
 u_1(q)+\cdots +u_{m-1}(q)&= {{\sigma_{1,1}-\sigma_{1,m}}\over {\sigma_{1,m}-1}}u_0(q)\\
  u_1(q)+\cdots +u_{m-1}(q)+u_m(q) &=-\sigma_{1,1}u_0(q).
\end{cases}
\end{equation}
Taking $u_0(q)=1$, from  (\ref{result1}) we obtain the precise expressions of all $u_l(q)$ in (\ref{line}), which are stated as follows:
\begin{equation}\label{result2}
\begin{cases}
u_0(q) & =1\\
u_1(q) &={{\sigma_{1,1}-\sigma_{1,2}}\over {\sigma_{1,2}-1}}\\
u_2(q) &= {{\sigma_{1,1}-\sigma_{1,3}}\over {\sigma_{1,3}-1}}-{{\sigma_{1,1}-\sigma_{1,2}}\over {\sigma_{1,2}-1}}={{(\sigma_{1,1}-1)(\sigma_{1,2}-\sigma_{1,3})}\over{(\sigma_{1,2}-1)(\sigma_{1,3}-1)}}\\
\ \ \ \vdots &\ \vdots\\
u_{m-1}(q) &= {{\sigma_{1,1}-\sigma_{1,m}}\over {\sigma_{1,m}-1}}-{{\sigma_{1,1}-\sigma_{1,m-1}}\over {\sigma_{1,m-1}-1}}={{(\sigma_{1,1}-1)(\sigma_{1,m-1}-\sigma_{1,m})}\over{(\sigma_{1,m-1}-1)(\sigma_{1,m}-1)}}\\
u_m(q) &= {{\sigma_{1,m}-\sigma_{1,1}}\over {\sigma_{1,m}-1}}-\sigma_{1,1}={{\sigma_{1,m}(1-\sigma_{1,1})}\over{\sigma_{1,m}-1}}.
\end{cases}
\end{equation}
Since all $\sigma_{1, i}$'s are distinct, it is easy to see that all $u_l(q)$'s are nonzero.
Putting  (\ref{result2}) into (\ref{line-a}),  we conclude that
\begin{equation}\label{line-a1}
\begin{cases}
b_1(q) &= \sigma_{1,1}b_0(q)\\
b_2(q) &= \sigma_{1,2}b_1(q)\\
\ \ \ \ \vdots &\  \vdots\\
b_{m-1}(q)&= \sigma_{1, m-1}b_{m-2}(q) \\
 b_m(q)&= \sigma_{1, m}b_{m-1}(q)
\end{cases}
\end{equation}
 which gives  cyclic equalities among $b_0(q), b_1(q), ..., b_m(q)$.
This implies that all $b_0(q)$, $b_1(q), ..., b_m(q)$ must be nonzero.
In fact, if one of them is zero, then the  cyclic equalities in (\ref{line-a1}) will induce that they are all zero. Moreover, we can yield  from the system (\ref{e-system2-1}) that ${\bf x}(q)={\bf 0}$, which is impossible.
\vskip .2cm
Now,  from the equation system (\ref{line-a1}) we easily read out a contradiction equation
$$b_1(q)=\sigma_{1,1}b_0(q)$$
 since $b_0(q)=\chi^0(M)$ and $b_1(q)=\chi^1(M)$ are only nonzero integers.

\vskip .2cm
It remains to prove that   the equation system (\ref{diag}) established beforehand  and all expressions of $u_l(q)$ in (\ref{result2})
are compatible with the system (\ref{line1}). %Here we only check the first two equations in the system (\ref{line1}), and we would like to leave other cases to readers as an exercise.

\vskip .2cm
First we  check the first  equation in the system (\ref{line1}), which has a little bit difference from the general case.
Since we have assumed that $\sigma_{1,1}u_0(q)+\sum\limits_{j=1}^mu_j(q)=0$, it suffices to check that
$$\sum_{i=2}^m \big(\sigma_{1,i}u_0(q)+\sum\limits_{j=1}^mu_j(q)\big)\varepsilon_if_i(q)$$
must be zero. Using the equations in (\ref{result1}), (\ref{result2}) and (\ref{line-a1}), we directly calculate as follows:
\begin{align*}
&\sum_{i=2}^m \big(\sigma_{1,i}u_0(q)+\sum\limits_{j=1}^mu_j(q)\big)\varepsilon_if_i(q)\\
=& (\sigma_{1,2}-\sigma_{1,1})\varepsilon_2f_2(q)+\cdots +(\sigma_{1,m}-\sigma_{1,1})\varepsilon_mf_m(q)\\
=& \sigma_{1,2}\varepsilon_2f_2(q)+\cdots +\sigma_{1,m}\varepsilon_mf_m(q)
-\sigma_{1,1}(\varepsilon_2f_2(q)+\cdots+\varepsilon_mf_m(q))\\
=& \sigma_{1,2}\varepsilon_2f_2(q)+\cdots +\sigma_{1,m}\varepsilon_mf_m(q)-\sigma_{1,1}(b_0(q)-\varepsilon_1f_1(q))\\
=&\sigma_{1,1}\varepsilon_1f_1(q))+ \sigma_{1,2}\varepsilon_2f_2(q)+\cdots +\sigma_{1,m}\varepsilon_mf_m(q)-\sigma_{1,1}b_0(q)\\
=& b_1(q)-\sigma_{1,1}b_0(q)\\
=&0
\end{align*}
as desired.
\vskip .2cm
 Next let us check the $l$-th equation in the system (\ref{line1}), where $l\geq 2$.
 We need to show that
 \begin{align*}
\sum_{i\not=l}\Big(\sigma^l_{1,i}\sum\limits_{j=0}^{l-1}u_j(q)+\sigma^{l-1}_{1,i}\sum\limits_{j=l}^mu_j(q)\Big)
\varepsilon_if_i(q)\\
\end{align*}
 is equal to zero. From the equations in  (\ref{result2}), we see easily that
$$\sum\limits_{j=0}^{l-1}u_j(q)=1+{{\sigma_{1,1}-\sigma_{1,l}}\over {\sigma_{1,l}-1}}={{\sigma_{1,1}-1}\over {\sigma_{1,l}-1}}$$
and $$\sum\limits_{j=l}^mu_j(q)=-\sigma_{1,1}-{{\sigma_{1,1}-\sigma_{1,l}}\over {\sigma_{1,l}-1}}=-
{{\sigma_{1,l}(\sigma_{1,1}-1)}\over {\sigma_{1,l}-1}}.$$
So
\begin{align*}
&\sum_{i\not=l}\Big(\sigma^l_{1,i}{{\sigma_{1,1}-1}\over {\sigma_{1,l}-1}}-\sigma^{l-1}_{1,i}{{\sigma_{1,l}(\sigma_{1,1}-1)}\over {\sigma_{1,l}-1}}\Big)
\varepsilon_if_i(q)\\
=& {{\sigma_{1,1}-1}\over {\sigma_{1,l}-1}}\Big(\sigma_{1,1}^l\varepsilon_1f_1(q)+\cdots +\sigma_{1,l-1}^l\varepsilon_{l-1}f_{l-1}(q)+\sigma_{1,l+1}^l\varepsilon_{l+1}f_{l+1}(q)+\cdots
+\sigma_{1,m}^l\varepsilon_{m}f_{m}(q)\Big)\\
&- {{\sigma_{1,l}(\sigma_{1,1}-1)}\over {\sigma_{1,l}-1}}\Big(\sigma_{1,1}^{l-1}\varepsilon_1f_1(q)+\cdots +\sigma_{1,l-1}^{l-1}\varepsilon_{l-1}f_{l-1}(q)+\sigma_{1,l+1}^{l-1}\varepsilon_{l+1}f_{l+1}(q)+\cdots\\
&+\sigma_{1,m}^{l-1}\varepsilon_{m}f_{m}(q)\Big)\\
=& {{\sigma_{1,1}-1}\over {\sigma_{1,l}-1}}\big(b_l(q)-\sigma^l_{1, l}\varepsilon_lf_l(q)\big)
-{{\sigma_{1,l}(\sigma_{1,1}-1)}\over {\sigma_{1,l}-1}}\big(b_{l-1}(q)-\sigma_{1,l}^{l-1}\varepsilon_lf_l(q)\big)\\
=& {{\sigma_{1,1}-1}\over {\sigma_{1,l}-1}}\big(b_{l}(q)-\sigma_{1,l}b_{l-1}(q)\big)\\
=&0
\end{align*}
as desired.
\vskip .2cm
Together with   all arguments as above, we complete the proof.
\end{proof}

\begin{proof}[{Proof of Theorem~\ref{main}}]
By Lemma~\ref{key}, we must have
$$m\geq k+1=[{n\over 2}]+1>{n\over 2}$$
as desired. This completes the proof of Theorem~\ref{main}.
\end{proof}

\begin{corollary}
The coefficient matrix of the system (\ref{e-system1}) has rank $k+1$.
\end{corollary}

\section{An observation}\label{ob}
Suppose that $M^{2n}$ is a unitary $S^1$-manifold fixing isolated  points, where $M^{2n}$ is not necessarily assumed to be nonbounding.

\vskip .2cm
First assume that $w_{i,1}, ..., w_{i, |{\bf w}_i^-|}$ are all negative weights in the weight set ${\bf w}_i$. Now, using the same way as in \cite[(*), page 172]{Ko1},  we change the expression of $\chi^l(M)$   into
\begin{align}\label{chi-1}
\chi^l(M)=\sum_{i=1}^m\varepsilon_i(-1)^{|{\bf w}_i^-|}\frac{\sigma_l(q^{w_{i,1}},..., q^{w_{i,n}})q^{-\sum\limits_{j\leq |{\bf w}_i^-|}w_{i,j}}}{\prod\limits_{j\leq |{\bf w}_i^-|}(1-q^{-w_{i,j}})\prod\limits_{j>|{\bf w}_i^-|}(1-q^{w_{i,j}})}.
\end{align}
Since $\chi^l(M)$ is an integer, to calculate $\chi^l(M)$ is enough to determine the constant term in the right side of (\ref{chi-1}).
Obviously, only those terms with $|{\bf w}_i^-|=l$ in the right side of (\ref{chi-1}) can produce the constant term.  An easy argument shows that
$$\chi^l(M)=(-1)^l(n_l^+-n_l^-)$$
        where  $n_l^+$ (resp. $n_l^-$) denotes the number of those fixed points $p_i$ with $|{\bf w}_i^-|=l$ and $\varepsilon_i=+1$ (resp. $\varepsilon_i=-1$).
Then the top Chern number
$$c_n[M]=T_{1,-1}(M)=\sum_{l=0}^n(-1)^l\chi^l(M)=\sum_{l=0}^n (n_l^+-n_l^-).$$
On the other hand, since the Euler characteristic $\chi(M)$ is exactly equal to the number of all fixed points, we have
$$\chi(M)=\sum_{l=0}^n (n_l^++n_l^-).$$
Furthermore, we conclude that
\begin{equation}\label{chern-euler}
\chi(M)=c_n[M]+2\sum_{l=0}^n n_l^-
\end{equation}
and
\begin{equation}\label{chern-euler1}
\chi(M)+c_n[M]=2\sum_{l=0}^n n_l^+.
\end{equation}
 As in the case of almost complex, we see that all signs $\varepsilon_i$ of fixed points  are positive  if and only if   $\chi(M)=c_n[M]$. A  classical result of Thomas \cite{T} tells us that $\chi(M^{2n})=c_n[M]$
is sufficient for a unitary manifold to be almost complex. Whether $\chi(M^{2n})=c_n[M]$ holds or not
completely does depend upon the choices of all signs $\varepsilon_i$ of fixed points. It seems to be not easy to determine when $\varepsilon_i$ is chosen as $+1$ or $-1$. Both (\ref{chern-euler}) and (\ref{chern-euler1}) only tell us that
$\chi(M)\equiv c_n[M]\ \mod\ 2$. As a result, it is stated as follows.

 \begin{corollary}
 Let $M$ be a unitary $S^1$-manifold fixing isolated  points.  Then $\chi(M)$ and $c_n[M]$
 have the same parity.
 \end{corollary}
As a consequence,
it is not difficult to see that  when $\chi(M)=2$, both (\ref{chern-euler}) and (\ref{chern-euler1}) force all $n_l^-$ to be zero, so $\chi(M)=c_n[M]$
and then $M$ is almost complex, as seen in \cite{Ko1, Ko2}.

\vskip .2cm

\noindent {\bf Acknowledgements.} The author would like to
thank Taras Panov for the information of the reference~\cite{T}, and also to thank
Wei Wang for his interest.

\footnotesize
\bibliographystyle{abbrv}

\end{document}